\documentclass[a4paper,12pt]{amsart}

\usepackage{amssymb}
\usepackage{latexsym}
\usepackage{amsmath}
\usepackage{amscd}
\usepackage{graphicx}

\title[Symmetry of derivation Lie algebras]
{Symmetry of symplectic derivation Lie algebras of free Lie algebras}

\author{Shigeyuki Morita}
\address{Graduate School of Mathematical Sciences, 
The University of Tokyo, 
3-8-1 Komaba Meguro-ku Tokyo 153-8914, Japan}
\email{morita@ms.u-tokyo.ac.jp}
\author{Takuya Sakasai}
\address{Graduate School of Mathematical Sciences, 
The University of Tokyo, 
3-8-1 Komaba Meguro-ku Tokyo 153-8914, Japan}
\email{sakasai@ms.u-tokyo.ac.jp}
\author{Masaaki Suzuki}
\address{Department of Frontier Media Science,  
Meiji University, 
4-21-1 Nakano, Nakano-ku, Tokyo  164-8525, Japan}
\email{macky@fms.meiji.ac.jp}


\newtheorem{thm}{Theorem}[section]
\newtheorem{prop}[thm]{Proposition}
\newtheorem{lem}[thm]{Lemma}

\theoremstyle{definition}

\newtheorem{remark}[thm]{Remark}

\begin{document}

\newcommand{\Mg}{\mathcal{M}_g}
\newcommand{\Mgp}{\mathcal{M}_{g,\ast}}
\newcommand{\Mgb}{\mathcal{M}_{g,1}}

\newcommand{\hg}{\mathfrak{h}_{g,1}}
\newcommand{\ag}{\mathfrak{a}_g}
\newcommand{\Ln}{\mathcal{L}_n}

\newcommand{\Sg}{\Sigma_g}
\newcommand{\Sgb}{\Sigma_{g,1}}
\newcommand{\la}{\lambda}

\newcommand{\Symp}[1]{Sp(2g,\mathbb{#1})}
\newcommand{\symp}[1]{\mathfrak{sp}(2g,\mathbb{#1})}
\newcommand{\gl}[1]{\mathfrak{gl}(n,\mathbb{#1})}

\newcommand{\At}[1]{\mathcal{A}_{#1}^t (H)}
\newcommand{\Hq}{H_{\mathbb{Q}}}

\newcommand{\Ker}{\mathop{\mathrm{Ker}}\nolimits}
\newcommand{\Hom}{\mathop{\mathrm{Hom}}\nolimits}
\renewcommand{\Im}{\mathop{\mathrm{Im}}\nolimits}

\newcommand{\Der}{\mathop{\mathrm{Der}}\nolimits}
\newcommand{\Out}{\mathop{\mathrm{Out}}\nolimits}
\newcommand{\Aut}{\mathop{\mathrm{Aut}}\nolimits}
\newcommand{\Q}{\mathbb{Q}}
\newcommand{\Z}{\mathbb{Z}}

\begin{abstract}
We show that a certain symmetry exists in the stable
irreducible decomposition of the Lie algebra consisting
of symplectic derivations of the free Lie algebra generated
by the first homology group of compact oriented surfaces.
\end{abstract}

\renewcommand\baselinestretch{1.1}
\setlength{\baselineskip}{16pt}

\newcounter{fig}
\setcounter{fig}{0}

\maketitle

\section{Introduction and statement of the main result}\label{sec:intro}

Let $\Sigma_{g,1}$ be a compact oriented surface of genus $g\geq 1$ 
with one boundary component and 
we denote $H_1(\Sigma_{g,1};\Q)$ simply by $H$. 
Equipped with the 
skew symmetric bilinear form induced by the intersection 
pairing, the space $H$ can be regarded as the standard symplectic 
vector space of dimension $2g$. Let $\mathcal{L}_H$ be the 
free Lie algebra generated by $H$ and let $\mathcal{L}_H(k)$ 
be the degree $k$ homogeneous part of it with respect to the 
natural grading. Let $\mathfrak{h}_{g,1}$ be the Lie algebra 
consisting of {\it symplectic} derivations of $\mathcal{L}_H$, 
which we call the symplectic derivation Lie algebra.
It has a natural grading and the degree $k$ part can be 
expressed as
\begin{equation}
\label{eq:sd}
\begin{split}
\mathfrak{h}_{g,1}(k)&=
\left\{D\in \mathrm{Hom}(H,\mathcal{L}_H(k+1)); D(\omega_0)=0\right\}\\
&\cong \mathrm{Ker}\left(H\otimes \mathcal{L}_H(k+1)
\overset{[\ ,\ ]}{\longrightarrow}\mathcal{L}_H(k+2)\right)
\end{split}
\end{equation}
where $\omega_0\in \mathcal{L}_H(2)\cong\wedge^2 H$ denotes
the symplectic class and $[\ ,\ ]$ denotes the bracket
operation on the Lie algebra $\mathcal{L}_H$.
The second isomorphism is induced by the Poincar\'e
duality $H^*\cong H$ (see \cite{morita93} for details).
The Lie algebra $\mathfrak{h}_{g,1}$ plays a fundamental role
first in the theory of Johnson homomorphisms and
more recently in the cohomological study of the outer
automorphism group $\mathrm{Out}F_n$ of free groups
which was initiated by the Lie version of the theory
of graph homology due to Kontsevich \cite{kontsevich1}\cite{kontsevich2}. 
We refer to \cite{morita06} for a recent survey and
also to more recent papers
including  \cite{es} where Enomoto and Satoh 
found a new series in the cokernel of the
Johnson homomorphisms, our former paper \cite{mss1},
and \cite{ckv} where Conant, Kassabov and Vogtmann
present a remarkable new development.

Let $\mathrm{Sp}(2g,\Q)$ be the symplectic group which
we sometimes denote simply by $\mathrm{Sp}$. If we fix a symplectic 
basis of $H$, then it can be 
considered as the standard representation of $\mathrm{Sp}(2g,\Q)$.
Each piece $\mathfrak{h}_{g,1}(k)$ is naturally an $\mathrm{Sp}$-module so that it has an irreducible decomposition. It is 
well known that this decomposition stabilizes when
$g$ is sufficiently large.

The purpose of this note is to show that a certain symmetry exists 
in the structure of $\mathfrak{h}_{g,1}(k)$. 
To describe our symmetry, we have to 
consider $\mathfrak{h}_{g,1}(k)$
as a $\mathrm{GL}(2g,\Q)$-module rather than an
$\mathrm{Sp}$-module. We define this structure by the identification
of $\mathfrak{h}_{g,1}(k)$ with the {\it second\/} module
in \eqref{eq:sd} which has a natural structure of
a $\mathrm{GL}(2g,\Q)$-module because the bracket operation
is clearly a $\mathrm{GL}(2g,\Q)$-morphism.
It follows that each $\mathfrak{h}_{g,1}(k)$ 
has an irreducible decomposition as a
$\mathrm{GL}(2g,\Q)$-module and this decomposition 
also stabilizes when $g$ is sufficiently large.
This stable decomposition can be described as
a linear combination of Young diagrams with 
$(k+2)$ boxes. The original decomposition as 
an $\mathrm{Sp}$-module can be obtained by
applying the restriction law associated 
with the pair $(\mathrm{GL}(2g,\Q), \mathrm{Sp}(2g,\Q))$.

To describe our result, we need one more technical 
term from the theory of representations of 
symmetric groups (see \cite{fh} for example). 
Any Young diagram $\la$ with
$k$ boxes defines an irreducible representation 
$V_\la$ of the symmetric group $\mathfrak{S}_k$.
The {\it conjugate} Young diagram of $\la$,
denoted by $\la'$, is the one obtained by
interchanging the roles of rows and columns. 
For example $[4^21^4]'=[62^3]$. As is well known,
the associated representation $V_{\la'}$ is isomorphic
to the tensor product of $V_\la$ with the alternating
representation $V_{[1^k]}$.

Now we can state our main result.

\begin{thm}\label{thm:s}
Let $\mathfrak{h}_{g,1}$ be the symplectic derivation
Lie algebra of the free Lie algebra $\mathcal{L}_H$ and
let $\mathfrak{h}_{g,1}(k) 
=\mathrm{Ker} \big( [\,\ ,\ ]:H\otimes \mathcal{L}_H(k+1) 
\to \mathcal{L}_H(k+2)\big)$ be
the degree $k$ part of it which we understand as
a $\mathrm{GL}(2g,\Q)$-module.

Assume that $k\equiv 2\ \text{or}\ 3\ \mathrm{mod} \ 4$.
Then the stable irreducible decomposition of
$\mathfrak{h}_{g,1}(k)$, 
which is expressed 
as a linear combination of Young
diagrams with $(k+2)$ boxes,
is symmetric with respect to taking the
conjugate Young diagrams. 
\end{thm}

The first two cases are given by 
$\mathfrak{h}_{g,1}(2)=[2^2],
\mathfrak{h}_{g,1}(3)=[31^2]$
(see \cite{hain} \cite{an}) both of which are
symmetric (i.e. self-conjugate) Young diagrams.
Here, for simplicity, we use the same symbol of Young diagram 
for the corresponding irreducible $\mathrm{GL}(2g,\Q)$-module.
The summand $[2^2]$ plays a very important role
in Hain's infinitesimal presentation of the Torelli
groups while the summand $[31^2]$ is the first
place where the trace maps introduced in
\cite{morita93} appear.
In higher degrees, there occur plural irreducible
components. In these cases, our result shows that the multiplicity $m_\la$
corresponding to a Young diagram $\la$ is equal to the multiplicity $m_{\la'}$ corresponding to 
the conjugate Young diagram $\la'$.

\begin{remark}
By combining the technique of this paper with that of
\cite{morita11}, we obtain a structure theorem for 
the Lie subalgebra
$\mathfrak{h}^{\mathrm{Sp}}_{g,1}\subset
\mathfrak{h}_{g,1}$
consisting of $\mathrm{Sp}$-invariant elements.
More precisely, we obtain an orthogonal
basis, with respect to a certain canonical metric, 
for each degree $2k$ part $\mathfrak{h}^{\mathrm{Sp}}_{g,1}(2k)$
which is parametrized by the set of partitions
of $k$ (equivalently the set of Young diagrams with
$k$ boxes)
each element of which represents
an eigenspace of mutually different eigenvalue. 
This gives, in particular, a complete description of
the degeneration of $\mathfrak{h}^{\mathrm{Sp}}_{g,1}(2k)$
from the stable range of $g$ one by one
down to the case of genus one
which corresponds to the eigenspace with the {\it largest}
eigenvalue.
This orthogonal direct sum decomposition
should be useful in the study of the arithmetic mapping class group, cohomology of $\mathrm{Out} F_n$ and
eventually the cohomology of the group of homology cobordism
classes of homology cylinders.
Details are discussed in \cite{mss2}.
\end{remark}

\begin{remark}
Besides the theoretical meaning, our main theorem 
can be used to make explicit computer calculations
concerning the structure of the Lie algebra $\hg$ 
as well as other related Lie algebras
considerably easier. 
See our paper \cite{mss1a} for details.
In fact, we noticed this symmetry
when we tried to overcome memory problems which arose 
in our earlier computer calculations.
%
%
%
%
%
%
With the help of this technique, 
we have determined the irreducible decomposition of 
$\mathfrak{h}_{g,1}(k)$ for all $k\leq 20$. 
For example, the dimensions of $\mathfrak{h}^{\mathrm{Sp}}_{g,1}(18)$
and
$\mathfrak{h}^{\mathrm{Sp}}_{g,1}(20)$ are given 
by the following tables. 
\begin{table}[h]
\begin{center}
\begin{tabular}{|c|r|r|r|r|r|r|r|r|r|}
\noalign{\hrule height0.8pt}
\hfil {} & $g=1$ & $g=2$ & $g=3$ & $g=4$ & $g=5$  \\
\hline
$\mathfrak{h}_{g,1}(18)^{\mathrm{Sp}}$ & $57$ & $100908$ & $888099$ & $1548984$ & 
$1710798$ \\
\hline
$\mathfrak{h}_{g,1}(20)^{\mathrm{Sp}}$ & $108$ & $869798$ & $12057806$ & $25062360$ & $29129790$ \\
\noalign{\hrule height0.8pt}
\end{tabular}
\end{center}
\label{tab:i1}
\end{table}

\begin{table}[h]
\begin{center}
\begin{tabular}{|c|r|r|r|r|r|r|r|r|r|}
\noalign{\hrule height0.8pt}
\hfil {} & $g=6$ & $g=7$ & $g=8$ & $g \geq 9$   \\
\hline
$\mathfrak{h}_{g,1}(18)^{\mathrm{Sp}}$ & $1728591$ & $1729620$ & $1729656$ & $1729657$ \\
\hline
$\mathfrak{h}_{g,1}(20)^{\mathrm{Sp}}$ & $29688027$ & $29728348$ & $29729957$ & $29729988$ \\
\noalign{\hrule height0.8pt}
\end{tabular}
\end{center}
\label{tab:i2}
\end{table}

\end{remark}

\par
\noindent
{\it Acknowledgement}\
The authors would like to thank Naoya Enomoto for 
informing them that Proposition 3.3 was already proved
by Zhuravlev in \cite{z}. He points out that,
our main theorem also follows from this by applying Pieri's formula on the middle
term of the exact sequence \eqref{eq:ses} below. 
Also, the authors would like to thank the referee for his/her careful reading and giving 
helpful comments to this paper. 
The authors were partially supported by KAKENHI (No.~24740040,  
No.~24740035, and 15H03618), 
Japan Society for the Promotion of Science, Japan.

\section{Character formulae}\label{sec:character}

In \cite{kontsevich1}\cite{kontsevich2}, Kontsevich described
the following result.

\begin{thm}[Kontsevich {\cite[Theorem 3.2]{kontsevich2}}]\label{thm:k}
Let $W_k$ be the $\mathfrak{S}_{k+2}$-module with character
\begin{align*}
\chi(1^{k+2})&=k!, \\
\chi(1^1a^b)&=(b-1)! a^{b-1} \mu(a) \quad (a \ge 2, b \ge 1, ab=k+1),\\
\chi(a^b)&=-(b-1)! a^{b-1} \mu(a) \quad (a \ge 2, b \ge 1, ab=k+2),
\end{align*}
\noindent
and $\chi$ vanishes on all the other conjugacy classes, 
where $\mu$ denotes the M\"obius function. Then there exists an isomorphism
$$
\mathfrak{h}_{g,1}(k)\cong H^{\otimes (k+2)}\otimes_{A_{k+2}} W_k
$$
of\/ $\mathrm{Sp}(2g,\Q)$-modules, where $A_{k+2}=\Q\mathfrak{S}_{k+2}$
denotes the group algebra of $\mathfrak{S}_{k+2}$.
\end{thm}

Kontsevich mentioned that for the proof of the above theorem
he developed a language of non-commutative geometry but
explicit description was not given. 
Here we would like to present an argument, 
which we noticed some time ago, showing that the
above theorem can be deduced directly from a
classical result in the theory of free Lie algebras. We mention 
that a combinatorial treatment of the above theorem of Kontsevich 
is given in \cite[Remark 7.5]{es}. 


Let $F_n$ be a free group of rank $n\geq 2$ and 
let $\mathcal{L}_n$ be the free Lie algebra generated 
by $H_1(F_n;\Q)\cong \Q^n$. We denote by
$\mathcal{L}_{n}(k)$
the degree $k$ part with respect to the natural grading. 
When $n=2g$, we can identify $H_1 (F_n;\Q)$ and $\mathcal{L}_n(k)$ with 
$H$ and $\mathcal{L}_H(k)$ as $\mathrm{GL}(n,\Q)$-modules. 

The classical character formula 
for the $\mathrm{GL}(n,\Q)$-module 
$\mathcal{L}_{n}(k)$ (\cite[Theorem 8-3]{r}) 
can be phrased, 
somewhat differently from the usual expression, 
as follows.

\begin{thm}[Witt-Brandt, see {\cite{r}}]\label{thm:wb}
Let $L_k$ be the $\mathfrak{S}_{k}$-module with character
\begin{align*}
\chi(1^{k})&=(k-1)!, \\
\chi(a^b)&=(b-1)! a^{b-1} \mu(a) \quad (a \ge 2, b \ge 1, ab=k),
\end{align*}
\noindent
and $\chi$ vanishes on all the other
conjugacy classes. Then there exists an isomorphism
$$
\mathcal{L}_n(k)\cong H_1(F_n;\Q)^{\otimes k}\otimes_{A_k} L_k
$$
as a $\mathrm{GL}(n,\Q)$-module.
\end{thm}

\begin{proof}[Proof of Theorem $\ref{thm:k}$]
Recall that we identify $\mathfrak{h}_{g,1}(k)$
with the $\mathrm{GL}(2g,\Q)$-module described in the second
line of \eqref{eq:sd}. It follows that we have a short
exact sequence 
\begin{equation}
\label{eq:ses}
0\longrightarrow 
\mathfrak{h}_{g,1}(k)
\longrightarrow
H\otimes \mathcal{L}_H(k+1)
\longrightarrow
\mathcal{L}_H(k+2)
\longrightarrow
0
\end{equation}
of $\mathrm{GL}(2g,\Q)$-modules.
By Theorem \ref{thm:wb}, we have an isomorphism
$$
\mathcal{L}_H(k+2)\cong H^{\otimes (k+2)} \otimes_{A_{k+2}} 
L_{k+2}
$$
and the character of $L_{k+2}$, denoted by 
$\chi_{k+2}$, is given by
\begin{equation}
\label{eq:chi}
\begin{split}
\chi_{k+2}&(1^{k+2})=(k+1)!, \quad
\chi_{k+2}(a^b)=(b-1)! a^{b-1} \mu(a), \\
&\chi_{k+2}(\text{other conjugacy class})=0.
\end{split}
\end{equation}
On the other hand, it can be seen that the tensor
product $H\otimes \mathcal{L}_H(k+1)$ corresponds to
the induced representation
$\mathrm{Ind}^{\mathcal{S}_{k+2}}_{\mathcal{S}_{k+1}} L_{k+1}$.
By the well-known formula for the character of induced
representations (see e.g. \cite{fh}), we can deduce that
the character of this induced representation,
denoted by $\chi_{k+1}^{k+2}$, is given by
\begin{equation}
\label{eq:chii}
\begin{split}
\chi_{k+1}^{k+2}&(1^{k+2})=(k+2) k!, \quad
\chi_{k+1}^{k+2}(1^1a^b)=(b-1)! a^{b-1} \mu(a), \\
&\chi_{k+1}^{k+2}(\text{other conjugacy class})=0.
\end{split}
\end{equation}
Here the number $k+2$ appearing in the upper line 
corresponds to the index of the subgroup 
$\mathfrak{S}_{k+1}\subset \mathfrak{S}_{k+2}$.
Now in view of the exact sequence \eqref{eq:ses}, the module 
$\mathfrak{h}_{g,1}(k)$ corresponds to the 
virtual representation
$\mathrm{Ind}^{\mathcal{S}_{k+2}}_{\mathcal{S}_{k+1}} L_{k+1}
-L_{k+2}$ whose character is given by the difference
\eqref{eq:chii}$-$\eqref{eq:chi}.
This is the required formula and the proof is complete.
\end{proof}

\section{Proof of the main result and an application}\label{sec:proof}

To prove the main theorem, we prepare the following
simple technical lemma.

\begin{lem}
\label{lem:tech}
Let $c$ be a positive integer which is not congruent
to $2$ $\mathrm{mod}\ 4$. Then for any decomposition
$$
c=ab\quad 
$$
with positive integers $a, b$, the following condition
holds:
$$
\text{$(\ast)$ if $a$ is even and $\mu(a)\not=0$, then $b$ is even.}
$$
\end{lem}

\begin{proof}
First of all, the equality $4m+2=2(2m+1)$ shows that
the condition on $c$ is necessary. Now assume the
required condition. If $c$ is odd, then $c$ is not divisible
by any even number so that $(\ast)$ clearly holds. 
Hence we can assume $c=4d$ for some $d$. If $c=ab$
such that $a$ is even and $\mu(a)\not=0$, then 
$a$ is not divisible by $4$ because otherwise $\mu(a)=0$.
It follows that $b$ is even and the proof is complete.
\end{proof}

\begin{proof}[Proof of Theorem $\ref{thm:s}$]
Let $\la=[\la_1,\cdots,\la_{h}]$ be a Young diagram
with $(k+2)$ boxes and let $V_\la$ be the corresponding
irreducible representation of the symmetric group
$\mathfrak{S}_{k+2}$. Then the multiplicity $m_\la$ of
$V_\la$ in $W_k$ is expressed as
$$
m_\la=\frac{1}{(k+2)!} \sum_{\gamma\in \mathfrak{S}_{k+2}}
\chi_k (\gamma)\chi_\la^V(\gamma)
$$
where $\chi_k$ is the character of $L_k$ given by (\ref{eq:chi}) and 
$\chi_\la^V$ denotes the character of $V_\la$.
On the other hand, by Theorem \ref{thm:k} the value
$\chi_k (\gamma)$ may be non-zero only in the cases
where the conjugacy class of $\gamma\in\mathfrak{S}_{k+2}$ 
is one of the following three types $1^{k+2}, 1^1a^b, a^b$.
In the latter two types, two positive integers $a, b$
satisfy the equality $ab=k+1$ or $ab=k+2$.
If $k$ satisfies the required condition, namely
if $k\equiv 2\ \text{or}\ 3\ \mathrm{mod} \ 4$,
then both $k+1$ and $k+2$ are 
{\it not} congruent to $2$ mod $4$. Therefore by
Lemma \ref{lem:tech}, we conclude that the 
two numbers $a,b$ satisfy one of the following
three conditions
\begin{align*}
& \mathrm{(i)}\ \text{$a$ is odd} \\
& \mathrm{(ii)}\ \text{$a$ is even and $\mu(a)=0$} \\
& \mathrm{(iii)}\ \text{$a$ is even and $b$ is even}.
\end{align*}
The sign of any odd cycle $a$ is $+1$ while 
the sign of any even cycle $a$ is $-1$. Therefore
if $\gamma$ satisfies $\mathrm{(i)}$
or $\mathrm{(iii)}$ of the above conditions,
then the sign of $\gamma$ is $+1$. Of course the sign
of $1^{k+2}$ is $+1$. We can now conclude the following.
If $\chi_k(\gamma)\not=0$, then the sign of $\gamma$,
which is the same as the character value $\chi_{[1^{k+2}]}^V(\gamma)$ of $\gamma$ on the alternating
representation $[1^{k+2}]$, is $+1$. It follows that we have
the equality
$$
\chi_\la^V(\gamma)=\chi_{\la'}^V(\gamma)
$$
on such element $\gamma$ where $\la'$ denotes the conjugate
Young diagram of $\la$. Summing up, we now conclude
$$
m_\la=m_{\la'}
$$
which finishes the proof.
\end{proof}

\begin{remark}
Under the same condition as in Theorem \ref{thm:s},
we have an isomorphism
$$
W_k\cong W_k\otimes V_{[1^{k+2}]}
$$
of $\mathfrak{S}_{k+2}$-modules.
\end{remark}

As for the case of the free Lie algebra,
we have the following symmetry which was
already proved in \cite{z}.
Here we describe a proof for the sake of 
self-containedness of this note.

\begin{prop}[Zhuravlev {\cite[Proposition 4.1]{z}}]
Assume that 
$k\equiv 0,\ 1\ \text{or}\ 3\ \mathrm{mod} \ 4$.
Then the stable irreducible decomposition of
$\mathcal{L}_{n}(k)$, which is expressed 
as a linear combination of Young
diagrams with $k$ boxes,
is symmetric with respect to taking the
conjugate Young diagrams. 
\end{prop}

\begin{proof}
The proof is given by replacing Theorem \ref{thm:k} with 
Theorem \ref{thm:wb} in the above argument. 
\end{proof}

As a sample application of our result, we show 
that a close relation exists between an earlier result of 
Asada and Nakamura in \cite{an} and one particular
result of Enomoto and Satoh in \cite{es}.
More precisely, it was proved in the former
paper that the $\mathrm{Sp}$-irreducible representation
corresponding to the Young diagram $[2k,2]$
occurs in $\mathfrak{h}_{g,1}(2k)$
with multiplicity one for any $k$. It is easy to see
that this representation is the restriction of
a unique $\mathrm{GL}(2g,\Q)$ irreducible representation
corresponding to the same Young diagram.
On the other hand,
it was proved in the latter paper that 
the $\mathrm{GL}(2g,\Q)$ irreducible representation
$[2^2,1^{4m}]$ occurs in $\mathfrak{h}_{g,1}(4m+2)$
with multiplicity one for any $m$.
 If we set $k=2m+1$ in the former case,
we see that both $\mathrm{GL}(2g,\Q)$ irreducible 
representations corresponding to the Young diagrams 
$[4m+2,2]$ and $[2^2,1^{4m}]$ occur in
$\mathfrak{h}_{g,1}(4m+2)$
with multiplicity one.
Observe that these two Young diagrams are conjugate 
to each other. Therefore these two results are 
equivalent in the framework of our main result.

We have explicit general results concerning
multiplicities of several types of Young diagrams
which occur in $\mathfrak{h}_{g,1}$
other than those treated in \cite{an}\cite{es}. 
However, here we omit them.

Finally we make a remark.

\begin{remark}\label{rem:ag}
In the case of the Lie algebra 
$\mathfrak{a}_g=\bigoplus_k \mathfrak{a}_g(k)$
consisting of symplectic derivations of 
free associative algebra generated by $H$
without unit, introduced by Kontsevich 
\cite{kontsevich1}\cite{kontsevich2},
it is easy to see that the 
$\mathrm{GL}(2g,\Q)$ stable irreducible decomposition of 
degree $k$ part $\mathfrak{a}_g(k)$
has the symmetry under taking conjugate
Young diagrams for any odd $k$. This is because the 
sign of the
cyclic permutation $(12\cdots k+2)\in \mathfrak{S}_{k+2}$
is $+1$ for any odd $k$.
For example, as $\mathrm{GL}(2g,\Q)$-modules, we have
\begin{align*}
\mathfrak{a}_g(1)&=[3]\oplus [1^3]\\
\mathfrak{a}_g(3)&=[5]\oplus 
[32]\oplus 
2[31^2]\oplus 
[2^21]\oplus 
[1^5]\\
\mathfrak{a}_g(5)&=[7]\oplus
2[52]\oplus
3[51^2]\oplus
2[43]\oplus
5[421]\oplus
2[41^3]\oplus
3[3^21]\oplus
3[32^2]\oplus
5[321^2]\\
& \quad \ \oplus
3[31^4]\oplus
2[2^31]\oplus
2[2^21^3]\oplus
[1^7] \\
\mathfrak{a}_g(7)&=[9]\oplus3[72]\oplus4[71^2]\oplus6[63]\oplus11[621]\oplus6[61^3]
\oplus4[54]\oplus18[531]\oplus14[52^2]\\
& \quad \ \oplus21[521^2]\oplus8[51^4]\oplus10[4^21] \oplus 18[432]\oplus24[431^2]
\oplus24[42^21]\oplus21[421^3]\\
& \quad \ \oplus6[41^5]\oplus6[3^3]\oplus18[3^221]
\oplus14[3^31^3] \oplus 10[32^3]\oplus18[32^21^2]\oplus11[321^4]\\
& \quad \ \oplus4[31^6] \oplus4[2^41]\oplus6[2^31^3]
\oplus3[2^21^5]\oplus[1^9]
\end{align*}
and it is easy to see that these satisfy the 
required symmetry. In this case also, we have determined 
the irreducible decomposition of $\mathfrak{a}_{g}(k)$ for all 
$k\leq 20$.
\end{remark}

\bibliographystyle{amsplain}

\begin{thebibliography}{30}

\bibitem{an}
M.~Asada, H.~Nakamura, 
\textit{On the graded quotient modules of mapping class groups of surfaces}, 
Israel.\ J.\ Math. 90 (1995) 93--113. 

\bibitem{ckv}
J.~Conant, M.~Kassabov, K.~Vogtmann, 
\textit{Hairy graphs and the unstable homology of
$\mathrm{Mod}(g,s)$, $\mathrm{Out}(F_n)$ and $\mathrm{Aut}(F_n)$}, 
 J.\ Topol. 6 (2013), 119--153. 

\bibitem{es}
N.~Enomoto, T.~Satoh, 
\textit{New series in the Johnson cokernels of the mapping class groups of surfaces},
Algebr.\ Geom.\ Topol. 14 (2014), 627--669.

\bibitem{fh}
W.~Fulton, J.~Harris, 
``Representation Theory", 
Graduate Texts in Mathematics 129, Springer Verlag, 1991.

\bibitem{hain}
R.~Hain, 
\textit{Infinitesimal presentations of the Torelli groups}, 
J.\ Amer. \ Math. \ Soc. 10 (1997) 597--651.

\bibitem{kontsevich1}
M.~Kontsevich, 
\textit{Formal (non)commutative symplectic geometry}, 
from: ``The Gel'fand Mathematical Seminars, 1990--1992'', 
Birkh\"auser, Boston (1993) 173--187. 

\bibitem{kontsevich2}
M.~Kontsevich, 
\textit{Feynman diagrams and low-dimensional topology}, 
from: ``First European Congress of Mathematics, Vol. II (Paris, 1992)'', 
Progr.\ Math. 120, Birkh\"auser, Basel (1994) 97--121.

\bibitem{morita93} S.~Morita, 
\textit{Abelian quotients of subgroups of the mapping class 
group of surfaces}, Duke Math.\ J. 70 (1993), 699--726.

\bibitem{morita06} S.~Morita, 
\textit{Cohomological structure of the mapping class group
and beyond}, in: ``Problems on Mapping Class Groups and related Topics", edited by Benson Farb, 
Proc.\ Sympos.\ Pure\ Math. 74 (2006), American Mathematical Society, 329--354.

\bibitem{morita11}
S.~Morita, 
\textit{Canonical metric on the space of symplectic invariant tensors and its applications}, 
preprint, arXiv:1404.3354.

\bibitem{mss1}
S.~Morita, T.~Sakasai, M.~Suzuki, 
\textit{Abelianizations of derivation Lie algebras of the free associative algebra and the free Lie algebra}, 
Duke Math.\ J. 162 (2013), 965--1002.

\bibitem{mss1a}
S.~Morita, T.~Sakasai, M.~Suzuki, 
\textit{Computations in formal symplectic geometry and characteristic classes of moduli spaces}, 
Quantum Topol. 6 (2015), 139--182.

\bibitem{mss2}
S.~Morita, T.~Sakasai, M.~Suzuki, 
\textit{Structure of symplectic invariant Lie subalgebras of symplectic derivation Lie algebras}, 
Adv.\ Math. 282 (2015), 291--334.

\bibitem{r}
C.~Reutenauer, 
``Free Lie Algebras", London Mathematical Society Monographs,
New Series 7, Oxford University Press 1993.

\bibitem{z}
V.~M.~Zhuravlev, 
\textit{A free Lie algebra as a module over the full linear group}, 
Mat.\ Sb. 187 (1996), 59--80; translation in Sb.\ Math. 187 (1996), 215--236.


\end{thebibliography}

\end{document}